\documentclass[smallextended,envcountsect,envcountsame]{svjour3}
\usepackage{graphicx}
\usepackage{mathrsfs}
\usepackage{times}
\usepackage{fix-cm}
\usepackage{amsmath}
\usepackage{amssymb}
\usepackage{latexsym}
\usepackage{dsfont}
\usepackage{xcolor}
\usepackage{cite}
\usepackage{enumitem}
\usepackage{appendix}
\usepackage{amsmath, amssymb}

\usepackage{bbm}
\usepackage{todonotes}

\smartqed  
\makeatletter
\def\namedlabel#1#2{\begingroup
	#2%
	\def\@currentlabel{#2}%
	\phantomsection\label{#1}\endgroup
}
\makeatother

\renewcommand{\H}{\mathcal{H}}
\newcommand{\R}{\mathbb{R}}
\newcommand{\Rex}{\overline{\mathbb{R}}}
\newcommand{\N}{\mathbb{N}}
\newcommand{\tto}{\rightrightarrows}
\newcommand{\gph}{\operatorname{gph}}
\newcommand{\dom}{\operatorname{dom}}

\renewcommand{\epsilon}{\varepsilon}

\begin{document}
	
	\title{Characterization of Strong Local Maximal Monotonicity through Graphical Derivatives
	
	\thanks{The author was supported by ANID Chile under grants CMM BASAL funds for Center of Excellence FB210005, Project ECOS230027, MATH-AMSUD 23-MATH-17, and ANID BECAS/DOCTORADO NACIONAL 21230802.}
 }
	\titlerunning{Characterization of Strong Local Maximal Monotonicity}        
	
	\author{Juan Guillermo Garrido
	}
	
	\institute{Juan Guillermo Garrido  \at Departamento de Ingenier\'ia Matem\'atica,  Universidad de Chile, Santiago, Chile. \email{ jgarrido@dim.uchile.cl } }
	
	\date{Received: date / Accepted: date}

	\maketitle
	
	\begin{abstract}
		This paper is devoted to study a characterization of (strong) local maximal monotonicity in terms of a property involving the graphical derivative of a set-valued mapping defined on a Hilbert space. As a consequence, a second-order characterization of variational convexity is provided without the assumption of subdifferential continuity.
		\keywords{Variational analysis \and  Local monotonicity \and Prox regular functions\and Variational convexity\and Generalized differentiation}
		 \subclass{49J52\and 49J53 \and 47H05}
	\end{abstract}
	
	\section{Introduction}
    Local monotonicity often appears in variational analysis. For instance, prox regular functions can be characterized via the local monotonicity of subdifferential mapping (see, e.g., \cite{Poliquin1996, MR2113863}). Recently, Rockafellar introduced the concept of variational convexity in finite-dimensional spaces \cite{MR3995335}, motivated by the local maximal monotonicity of subdifferential mappings. This concept of functions has gained significant attention in recent years, as it has enabled the generalization of numerical algorithms for solving nonsmooth optimization problems (see, e.g., \cite{MR4550949,MR3995335}). In \cite{MR4784083}, variational convexity is further studied in general Banach spaces.

    Generalized differentiation is a fundamental tool for studying the first-order variational properties of set-valued mappings. In particular, the graphical derivative is used to establish criteria for Lipschitzian properties such as metric regularity/subregularity and strong metric regularity (see, e.g., \cite{MR2252233, MR3108434, MR3705353}). Furthermore, the graphical derivative can be employed to extend numerical algorithms to nonsmooth settings, as in the case of Newton's method (see, e.g., \cite{MR4252735, aragon2024coderivative}). Additionally, the explicit computation of graphical derivatives for certain important cases is an active area of research, as it contributes to the study of the stability of solution mappings (see, e.g., \cite{MR3288139, MR3485980, MR4769808}). For instance, workable calculus rules for graphical derivatives are provided in \cite{MR3428695, MR3693740}. 

    In this paper, we aim to study a characterization of local maximal monotonicity via graphical derivatives for general set-valued mappings in Hilbert spaces, as well as its corresponding counterparts for strong monotonicity and hypomonotonicity. In particular, for strong local maximal monotonicity, we provide several characterizations, all of which depend on the lower-definiteness of the graphical derivative (or strict) but require different additional assumptions, such as local maximal hypomonotonicity or strong metric regularity.

    A parallel result using coderivatives can be found in \cite[Theorem 3.4]{MR3485980} (see also \cite[Theorem 6.3]{khanh2025local}). Previously, a related result was established in \cite[Theorem 3.3]{MR3843842}, where the tilt stability property for prox-regular functions is characterized as the lower definiteness of graphical derivative of the subdifferential mapping. Later, in \cite[Theorem 4.3]{gfrerer2025}, second-order characterizations of variational convexity for prox-regular functions were provided, and in the same line, in \cite{Rockafellar2025} are also studied these kind of results from a point-based perspective.

    The paper is organized as follows: In Section 2, we present the necessary mathematical background and notation for the subsequent sections, along with some preparatory results. Then, in Section 3, we establish the main result of the paper and provide examples to illustrate the necessity of our hypotheses. In Section 4, a characterization of variational convexity is developed through a condition for $f$-attentive graphical derivative of subdifferential mapping, avoiding subdifferential continuity. The remainder of the paper are some concluding remarks.
    
    \section{Preliminaries and Auxiliary Results}
	Throughout the work, $\H$ stands for a general Hilbert space endowed by inner product $\langle \cdot, \cdot\rangle$ and a norm $\|\cdot\|$. The open (resp. closed) ball centered at $a\in\H$ and radius $r>0$ is denoted by $\mathbb{B}_r(a)$ (resp. $\mathbb{B}_r[a]$) and the closed unit ball is simply denoted by $\mathbb{B}$. For a given set $S\subset \H$, the distance function from $x$ to $S$ is denoted by $x\mapsto d(x;S)$, where we use the convention $d(x;S) = \infty$ whenever $S = \emptyset$.
	
	Given a set-valued mapping $F\colon \H\tto \H$, we define its domain as $\dom F := \{x\in\H : F(x)\neq\emptyset\}$ and its graph as $\gph F := \{(x,y)\in\H\times\H : y\in F(x)\}$.
	
	Let us recall that  a function $f\colon [0,1]\to \H$ is said to be absolutely continuous if for every $\epsilon>0$, there is $\delta>0$ such that for all pairwise disjoint set of intervals $(]a_i,b_i[)_{i\in\N}\subset [0,1]$ with $\sum_{j\in\N}(b_j-a_j)<\delta$, one has $\sum_{j\in\N}\|f(b_j)-f(a_j)\|<\epsilon$. The following proposition can be found in \cite[Theorem 2.2.17]{MR2168068}.
	\begin{proposition}\label{abs-con-car}
		Consider a function $f\colon [0,1]\to \H$. Suppose that $f$ is absolutely continuous, then $f$ is derivable a.e., $f'\in L^1([0,1];\H)$ and $$f(t) - f(s) = \int_s^t f'(\tau)d\tau, \forall s,t\in [0,1].$$ 
	\end{proposition}
    Let $\mathcal{K}$ be a metric space. For a given set-valued mapping $F\colon \mathcal{K}\tto \H$ we recall the following (see, e.g., \cite{zbMATH07502906})
 $$\limsup_{x'\to x} F(x') := \{ y\in\H : \exists (x_n,y_n)\to (x,y)  \text{ s.t. }\forall n\in\N, (x_n,y_n)\in \gph F\},$$
 and
 \begin{equation*}
\liminf_{x'\to x} F(x') :=\left\{
  y\in\H \;\middle|\;
  \begin{aligned}
  & \forall (x_n)\to x, \exists \text{ subsequence }(n_k)\subset \N,\\
  & \exists (y_{n_k})\to y, \forall k\in\N,(x_{n_k},y_{n_k})\in \gph F
  \end{aligned}
\right\}.
\end{equation*}
	Some concepts of local monotonicity are recalled. A set-valued mapping $F\colon \H\tto \H$ is said to be \emph{locally monotone} around $(\bar x,\bar y)\in \gph F$ if there are neighborhoods $U,V$ of $\bar x,\bar y$ respectively such that 
	\begin{equation}\label{hypo01}
		\langle y-y',x-x' \rangle\geq 0, \ \forall (x,y),(x',y')\in \gph F\cap (U\times V).
	\end{equation}
	We also say that $F\colon \H\tto \H$ is \emph{strongly locally monotone} at $(\bar x,\bar y)$ if there is $\kappa>0$ such that $F-\kappa \text{Id}$ is locally monotone at $(\bar x,\bar y-\kappa \bar x)$. Let us say that $\kappa$ is the modulus of strong local monotonicity. We say $F\colon \H\tto \H$ is \emph{locally hypomonotone} at $(\bar x,\bar y)$ if there is $\kappa>0$ such that $F+\kappa \text{Id}$ is locally monotone at $(\bar x,\bar y+\kappa \bar x)$.
	
	The set-valued mapping $F$ is said to be \emph{locally maximal monotone} at $(\bar x,\bar y)$ when $F$ is locally monotone at $(\bar x,\bar y)$ and there is a neighborhood $U\times V$ of $(\bar x,\bar y)$ such that for every globally monotone operator $S\colon \H\tto\H$ such that $\gph F\cap U\times V\subset \gph S$ then $\gph F\cap U\times V= \gph S\cap U\times V$. Similarly, we say that $F$ is \emph{locally maximal hypomonotone} at $(\bar x,\bar y)$ if there is $\rho>0$ such that $F+\rho\text{Id}$ is locally maximal monotone at $(\bar x,\bar y+\rho\bar x)$.
	
	A different definition of maximality was introduced previously in \cite{MR1886225}, which for reflexive Banach spaces it becomes equivalent to the one given here (see, e.g., \cite[Theorem 8.5]{MR4769808}). 
    
		We say $F$ is \emph{strongly locally maximal monotone} at $(\bar x,\bar y)$ with modulus $\kappa>0$ (see, e.g., \cite[Definition 6.5]{MR4769808}) if there is a neighborhood $U\times V$ of $(\bar x,\bar y)$ such that $$\langle y_1-y_2,x_1-x_2\rangle\geq \kappa\|x_1-x_2\|^2, \forall (x_1,y_1),(x_2,y_2)\in (U\times V)\cap \gph F,$$ and for every monotone operator $S\colon \H\tto\H$ such that $U\times V\cap\gph F\subset \gph S$ we have $U\times V\cap\gph F=U\times V\cap\gph S$.
	
	The following proposition shows that strong local maximal monotonicity can be expressed in terms of local maximal monotonicity.
	\begin{proposition}\label{prop-slmm-lmm}
		Let $F\colon \H\tto\H$ be a set-valued mapping,  $(\bar x,\bar y)\in \gph F$ and $\kappa>0$. Then, $F$ is strongly locally maximal monotone at $(\bar x,\bar y)$ with modulus $\kappa$ if and only if $F-\kappa\text{Id}$ is locally maximal monotone at $(\bar x,\bar y-\kappa\bar x)$. 
	\end{proposition}
	\begin{proof}
		It follows from \cite[Proposition 6.7 \& Corollary 8.9]{MR4769808}. 
	\end{proof}

	We recall some generalized Lipschitz notions for set-valued mappings (see, e.g., \cite{MR4769808}). 
    \begin{definition}
        Let $F\colon \H\tto \H$ be a set-valued mapping with $(\bar x,\bar y)\in \gph F$. We say
	\begin{enumerate}
		\item [1)] $F$ is \emph{metrically regular} with modulus $\kappa$ at $(\bar x,\bar y)$ if there is a neighborhood $U\times V$ of $(\bar x,\bar y)$ such that
		\begin{equation}\label{met-reg-def}
			d(x;F^{-1}(y))\leq \kappa\cdot d(y;F(x)), \forall (x,y)\in U\times V.
		\end{equation}
		The infimum of $\kappa$ over all the open sets $U\times V$ satisfying \eqref{met-reg-def} is denoted by $\text{reg}(F;(\bar x,\bar y))$ and is called the modulus of metric regularity.
		\item [2)] $F$ is \emph{strongly metrically subregular} at $(\bar x,\bar y)$ if there are $\kappa>0$ and  a neighborhood $U$ of $\bar x$ such that $$ \|x-\bar x\|\leq \kappa\cdot d(\bar y;F(x)), \forall x\in U. $$
		\item [3)] $F$ is \emph{strongly metrically regular} at $(\bar x,\bar y)$ if there is a neighborhood $U\times V$ of $(\bar x,\bar y)$ joint with a Lipschitz function $\vartheta\colon V\to U$ such that $V\times U\cap \gph F^{-1} = \gph \vartheta$. 
	\end{enumerate}
    \end{definition}
	Given a set $C\subset \H$ and $x\in C$, the (Bouligand) tangent cone and the strict tangent cone (also known as paratingent cone) of $C$ at $x$ are defined (see, e.g., \cite{MR2458436}), respectively, by
	\begin{equation*}
	    \displaystyle T(x;C) := \limsup_{s\searrow 0} \frac{C - x}{s} \text{ and }
		\displaystyle T_\ast(x;C) := \limsup_{\substack{y\to x, y\in C\\ s\searrow 0
        }}\frac{C-y}{s}.
	\end{equation*}
	
    For a set-valued mapping $F\colon \H\tto \H$ and $(x,y)\in \gph F$, for $u\in \H$, we define the graphical derivative of $F$ at $(x,y)$ in the direction $u$ as \begin{equation*}
	    DF(x,y)(u) :=  \{v\in\H : (u,v)\in T((x,y);\gph F)\},
	\end{equation*}
	and the strict graphical derivative of $F$ at $(x,y)$ in the direction $u$ as 
    \begin{equation*}
        D_\ast F(x,y)(u) := \{v\in\H : (u,v)\in T_\ast((x,y);\gph F)\}.
    \end{equation*}
    The following definition corresponds to positive definiteness of set-valued mappings. We adopt the name given in \cite{aragon2024coderivative}.
	\begin{definition}\label{low-def}
		Let us consider $\rho\in\R$. We say that a set-valued mapping $G\colon \H\tto \H$ is \emph{$\rho$-lower-definite} if for all $(u,v)\in \gph G$, $\langle u,v\rangle\geq \rho\|u\|^2$. For a given set-valued mapping $F\colon \H\tto \H$, we say $DF$ is $\rho$-lower-definite on $A\subset\H\times \H$ if for all $(x,y)\in A\cap \gph F$, $DF(x,y)$ is $\rho$-lower-definite. When $A$ is the neighborhood of a given point $(\bar x,\bar y)\in \gph F$, we say $DF$ is $\kappa$-lower-definite around $(\bar x,\bar y)$.
	\end{definition}
	For a given real number $\kappa\in\R$, we define the $\kappa$-shifted operator $\mathsf{T}_\kappa\colon\H\times \H\to \H\times \H$ given by 
    \begin{equation}\label{T_k}
        \mathsf{T}_\kappa(x,y) = (x,y+\kappa x),
    \end{equation}
    which is linear and invertible.
	\begin{lemma}\label{sum-rule-gd}
		Consider a set-valued mapping $F\colon \H\tto \H$  such that $DF$ is $\kappa$-lower-definite for some $\kappa\in\R$ on some set $A\subset \H\times \H$. Then, for all $\gamma\in\R$, $D(F+\gamma\text{Id})$ is $(\kappa+\gamma)$-lower-definite on $\mathsf{T}_\gamma(A)$.
	\end{lemma}
	\begin{proof}
		 It is not difficult to see that for every $(x,y)\in \mathsf{T}_\gamma(A)\cap \gph (F+\gamma\text{Id})$
		\begin{equation*}
			D(F+\gamma\text{Id})(x,y) = DF(x,y-\gamma x) + \gamma\text{Id}.
		\end{equation*}
		Then, if $v\in D(F+\gamma\text{Id})(x,y)(u)$ it follows that $v-\gamma u\in DF(x,y-\gamma x)(u)$. Since $(x,y-\gamma x)\in A$, we have $DF(x,y-\gamma x)$ is $\kappa$-lower-definite, then $\langle v-\gamma u,u\rangle\geq \kappa\|u\|^2$ and it yields $\langle u,v\rangle\geq (\kappa + \gamma)\|u\|^2$, obtaining the desired conclusion.
	\end{proof}
    
	The following lemma states a chain rule for graphical derivative, between a locally Lipschitz mapping and an absolutely continuous function.
	\begin{lemma}\label{lemma_cotingent_lips}
		Let $U\subset \H$ be an open set, $F\colon U\tto \H$ be a set-valued mapping, which is locally Lipschitz on its domain and $x\colon [0,1]\to U$ be an absolutely continuous function such that $x([0,1])\subset \dom F$. Then, we have that for almost all $t\in [0,1]$  $$DF(x(t))(\dot{x}(t)) = \left\{ \frac{d}{dt}(F\circ x)(t) \right\}.$$
	\end{lemma}
	\begin{proof}
		See \cite[Theorem 5.3.2]{MR2458436}.    
	\end{proof}

    We recall the following result concerning strong metric subregularity, established under the assumption of the lower-definiteness of the graphical derivative.
	\begin{lemma}\label{st-met-sub-lemma}
		Suppose that $\H$ is finite-dimensional. Let us consider $F\colon \H\tto \H$ to be a set-valued mapping with $(\bar x,\bar y)\in \gph F$ such that $F$ has locally closed graph at $(\bar x,\bar y)$ and $DF(\bar x,\bar y)$ is $\rho$-lower-definite. Then $F$ is strongly metrically subregular at $(\bar x,\bar y)$. In particular, $\bar x$ is an isolated point in $F^{-1}(\bar y)$.
	\end{lemma}
	\begin{proof}
		It is a direct consequence of \cite[Theorem 4E.1]{MR3288139}.
	\end{proof}
    
    The next lemma states the lower-definiteness of the strict graphical derivative derived from the local monotonicity.
	\begin{lemma}\label{slm-low-def}
		Consider a set-valued mapping $F\colon \H\tto \H$ and $(\bar x,\bar y)\in \gph F$. Suppose that $F-\kappa\text{Id}$ is locally monotone at $(\bar x,\bar y-\kappa\bar x)$ for some $\kappa\in\R$. Then, $(x,y)\tto D_\ast F(x,y)$ is $\kappa$-lower-definite around $(\bar x,\bar y)$.
	\end{lemma}
	\begin{proof}
		We have there is a neighborhood $U\times V$ of $(\bar x , \bar y - \kappa\bar x)$ such that \eqref{hypo01} is satisfied on $\gph (F-\kappa\text{Id})\cap (U\times V)$. For $(x,y)\in \gph F\cap \mathsf{T}_{\kappa}(U\times V)$, taking any $d\in D_\ast F(x ,y)(z)$ for $z\in \H$, then there are sequences $(s_k)\subset \R_+$, $(z_k,d_k)$ and $(x_k,y_k)\in \gph F$ such that $s_k\searrow 0$, $(z_k,d_k)\to (z,d)$, $(x_k,y_k)\to (x,y)$ and $(x_k ,y_k) + s_k(z_k,d_k)\in \gph F$. Then, for all $k\in\N$ big enough, we have
		\begin{equation*}
			\langle (y_k + s_k d_k) - y_k,(x_k + s_k z_k) -x_k \rangle \geq \kappa\|x_k + s_k z_k - x_k\|^2.
		\end{equation*}
		It follows that $\langle d_k,z_k\rangle\geq \kappa\|z_k\|^2$, then taking the limit when $k\to \infty$, we get $\langle d,z\rangle\geq \kappa\|z\|^2$, so we conclude that $D_\ast F(x,y)$ is $\kappa$-lower-definite.
	\end{proof}
	
	An important example of locally maximal hypomonotone operator is given in the following proposition.
	\begin{proposition}\label{prop-lips-max-hypo}
		Consider a set-valued mapping $F\colon \H\tto \H$ with $(\bar x,\bar y)\in \gph F$. Suppose that $F^{-1}$ is strongly metrically regular at $(\bar y,\bar x)$. Then, $F$ is locally maximal hypomonotone.
	\end{proposition} 
	\begin{proof}
		Since $F^{-1}$ is strongly metrically regular at $(\bar y,\bar x)$, we have there is a neighborhood $U\times V$ of $(\bar x,\bar y)$ and a Lipschitz function $\vartheta\colon U\to V$ such that $\gph F\cap U\times V = \gph \vartheta$. Let $\sigma>0$ be the Lipschitz constant of $\vartheta$, then for all $x,y\in U$, 
		\begin{equation*}\label{eqn_lips_hypo}
			\langle \vartheta(y)-\vartheta(x),x-y\rangle\leq \| \vartheta(x)-\vartheta(y)\|\cdot \|x-y\|\leq \sigma\|x-y\|^2.
		\end{equation*}
		It follows directly that $F+\sigma\text{Id}$ is locally monotone at $(\bar x,\bar y + \sigma \bar x)$. Regarding \eqref{T_k}, consider $W := \mathsf{T}_{\sigma}(U\times V)$, which is an open neighborhood of $(\bar x,\bar y + \sigma \bar x)$. Let us take a monotone operator $S\colon \H\tto \H$ such that $W\cap \gph (F+\sigma\text{Id})\subset \gph S$. Then, consider $(x,y)\in W\cap \gph S$, we have $(x,y-\sigma x)\in U\times V$. Since $x\in U$, there is $\delta>0$ such that $\mathbb{B}_\delta(x)\subset U$, then for all $v\in \mathbb{B}$ and $\epsilon\in [0,\delta[$, $x+\epsilon v\in U$. Note that for all $z\in U$, $(z,\vartheta(z)+\sigma z)\in W\cap \gph S$, and given that $S$ is monotone, for all $v\in \mathbb{B}$ and $\epsilon\in ]0,\delta]$
		\begin{equation*}
			\langle x+\epsilon v-x,\vartheta(x+\epsilon v)+\sigma (x+\epsilon v) -y \rangle\geq 0.
		\end{equation*}
		It follows that
		\begin{equation*}
			\langle v,\vartheta(x+\epsilon v) + \sigma\epsilon v-(y-\sigma x)\rangle\geq 0.
		\end{equation*}
		Then, taking $\epsilon\searrow 0$, by virtue of the continuity of $\vartheta$, we have $$\langle v,\vartheta(x)-(y-\sigma x)\rangle\geq 0, \ \forall v\in \mathbb{B},$$ and it implies that $\vartheta(x) = y-\sigma x$, so $(x,y-\sigma x)\in \gph F$, thus $(x,y)\in \gph(F+\sigma\text{Id})$. It follows that $W\cap \gph S = W\cap \gph (F+\sigma\text{Id})$, then $F+\sigma \text{Id}$ is locally maximal monotone at $(\bar x,\bar y+\sigma \bar x)$, concluding the desired result.
	\end{proof}
    
	We also provide a local closedness of the graph of a local maximal hypomonotone operator.
	\begin{lemma}\label{local-closed}
		Let $F\colon \H\tto\H$ be a set-valued mapping and $(\bar x,\bar y)\in \gph F$. Suppose that $F$ is locally maximal hypomonotone at $(\bar x,\bar y)$, then $\gph F$ is locally closed at $(\bar x,\bar y)$.
	\end{lemma}
	
	\begin{proof}
		Let us consider $\kappa>0$, and a neighborhood $U\times V$ of $(\bar x,\bar y + \kappa \bar x)$ such that $B := \gph (F+\kappa\text{Id})\cap U\times V$ is monotone, and we also consider a maximal monotone extension of $B$, denoted by $G$ (see \cite[Theorem 20.21]{MR3616647}). The local maximality of $F$ says us $$\gph (F+\kappa\text{Id})\cap U\times V = \gph G\cap U\times V.$$ Since $G$ is maximal monotone (globally), we have $\gph G$ is closed (see \cite[Proposition 20.38 (iii)]{MR3616647}). Taking $W := \mathsf{T}_{\kappa}^{-1}(U\times V)$, which is a neighborhood of $(\bar x,\bar y)$, we have $$\gph F\cap W = \mathsf{T}_{\kappa}^{-1}(\gph G)\cap W.$$
		Since $\mathsf{T}_{\kappa}^{-1}(\gph G)$ is closed, it follows $\gph F$ is locally closed at $(\bar x,\bar y)$.
	\end{proof}
	
	\section{Characterization of Local Strong Maximal Monotonicity}
	
	The theorem presented below corresponds to a characterization of strongly local maximal monotonicity through first-order properties related with lower definiteness of graphical derivatives (see Definition \ref{low-def}). This result can be seen as a counterpart criteria through coderivatives (see, e.g., \cite[Theorem 3.4]{MR3485980}). First, we begin with an important lemma, which will be useful later.
	
	\begin{lemma}\label{lem-met-reg}
		Let $F\colon \H\tto \H$ be a set-valued mapping such that  $DF$ is $\kappa$-lower-definite around  $(\bar x,\bar y)\in \gph F$ with $\kappa>0$ and $F$ is strongly metrically regular at $(\bar x,\bar y)$. Then, 
		\begin{enumerate}
			\item [(a)] $F$ is strongly locally monotone at $(\bar x,\bar y)$ with modulus $\kappa$.
			\item [(b)] $\text{reg}(F;(\bar x,\bar y))\leq \frac{1}{\kappa}$.
		\end{enumerate}
	\end{lemma}
	
	\begin{proof}
		Since $F$ is strongly metrically regular, there are a neighborhood $U\times V$ of $(\bar x,\bar y)$ and a Lipschitz function $\vartheta \colon V\to U$ where $\gph F^{-1}\cap V\times U = \gph \vartheta$. By shrinking $V$, assume it is open and convex. Also, suppose that $DF(x,y)$ is $\kappa$-lower-definite for all $(x,y)\in U\times V\cap \gph F$.\\
		\emph{{\bf Claim 1}: For every $y\in V$ and $(u,v)\in \gph D\vartheta(y)$ we have $\langle v,u \rangle\geq \kappa\|v\|^2$.}\\
		\emph{Proof of Claim 1:} Take $y\in V$. Observe that, $$DF(\vartheta(y),y)^{-1} = DF^{-1}(y,\vartheta(y)) = D\vartheta(y).$$ 
		Since $DF$ is $\kappa$-lower-definite it follows $\langle u,v\rangle\geq \kappa\|v\|^2$.\\
		\emph{{\bf Claim 2}: $F$ is strongly locally monotone at $(\bar x,\bar y)$ with modulus $\kappa$.}\\
		\emph{Proof of Claim 2:} Take $y_1,y_2\in V$. Define $\alpha\colon t\mapsto ty_2 + (1-t)y_1$, we have that $\alpha([0,1])\subset V$ and in view of the Lipschitz continuity of $\vartheta$, we have $\vartheta\circ \alpha$ is absolutely continuous. By Lemma \ref{lemma_cotingent_lips}, we also have $D\vartheta(\alpha(t))(y_2-y_1) = \{\frac{d}{dt}(\vartheta\circ \alpha)(t)\}$ for a.e. $t\in [0,1]$, thus Claim 1 yields
		\begin{equation*}
			\left\langle \frac{d}{dt}(\vartheta\circ \alpha)(t),y_2-y_1 \right\rangle\geq \kappa\left\|\frac{d}{dt}(\vartheta\circ \alpha)(t)\right\|^2 \text{ for a.e. }t\in [0,1].
		\end{equation*}
		Then,
		\begin{equation*}
			\begin{aligned}
				\kappa\|\vartheta(y_2)-\vartheta(y_1)\|^2& = \kappa \|(\vartheta\circ \alpha)(1) - (\vartheta\circ \alpha)(0)\|^2\\
				&= \kappa \left\|\int_0^1\frac{d}{dt}(\vartheta\circ \alpha)(t)dt \right\|^2\\
				&\leq \kappa\int_0^1\left\|\frac{d}{dt}(\vartheta\circ \alpha)(t)\right\|^2dt\\
				&\leq \int_0^1 \left\langle \frac{d}{dt}(\vartheta\circ \alpha)(t),y_2-y_1 \right\rangle dt\\
				&= \langle  (\vartheta\circ \alpha)(1) - (\vartheta\circ \alpha)(0),y_2-y_1\rangle \\
				&= \langle \vartheta(y_2) - \vartheta(y_1),y_2-y_1\rangle,
			\end{aligned}
		\end{equation*}
		where Proposition \ref{abs-con-car} was used, joint with Hölder inequality. Now, consider $(x_i,y_i)\in \gph F\cap U\times V$ for $i\in \{1,2\}$. It follows that $x_i = \vartheta(y_i)$ for $i\in \{1,2\}$, then by the recently proved inequality, we have $\langle x_1-x_2,y_1-y_2\rangle\geq \kappa\|x_1-x_2\|^2$.\\
		\emph{{\bf Claim 3}: $\text{reg}(F;(\bar x,\bar y))\leq \frac{1}{\kappa}$.}\\
		\emph{Proof of Claim 3:} Take $y_1,y_2\in V$, by the inequality proved in Claim 2 for $\vartheta$ and Cauchy-Schwarz inequality, we have directly $\vartheta$ is $\frac{1}{\kappa}$-Lipschitz on $V$, implying that $\text{reg}(F;(\bar x,\bar y))\leq \frac{1}{\kappa}$.
	\end{proof}
	
        Now, we present our main result. It provides diverse characterizations describing local maximal strong monotonicity through graphical derivatives and some additional hypotheses. 
	\begin{theorem}\label{prop-s-m-r}
		Let $F\colon \H\tto\H$ be a set-valued mapping, and $(\bar x,\bar y)\in \gph F$. Suppose that $\kappa>0$. The following assertions are equivalent
		\begin{enumerate}
			\item [(a)] $F$ is strongly locally maximal monotone at $(\bar x,\bar y)$ with modulus $\kappa$.
			\item [(b)] $DF$ is $\kappa$-lower-definite around $(\bar x,\bar y)$ and $F$ is locally maximal hypomonotone at $(\bar x,\bar y)$.
			\item [(c)] $DF$ is $\kappa$-lower-definite around $(\bar x,\bar y)$ and $F$ is strongly metrically regular at $(\bar x,\bar y)$.
		\end{enumerate}
		Moreover, when $\H$ is finite dimensional, all the assertions are equivalent to
		\begin{enumerate}
			\item [(d)] $D_\ast F$ is $\kappa$-lower-definite around $(\bar x,\bar y)$, $\gph F$ is locally closed at $(\bar x,\bar y)$ and $\bar x\in \liminf_{y\to\bar y}F^{-1}(y)$.
		\end{enumerate}
	\end{theorem}
	\begin{proof} The proof of the equivalence of \emph{(a)}, \emph{(b)} and \emph{(c)} will be completed by demonstrating \emph{(a)}$\Rightarrow$\emph{(b)}, \emph{(b)}$\Rightarrow$\emph{(c)} and \emph{(c)}$\Rightarrow$\emph{(a)}.\\
		\underline{\emph{(a)}$\Rightarrow$\emph{(b)}}: The lower definiteness of $DF$ follows from Lemma \ref{slm-low-def} and the local maximal hypomonotonicity is direct.\\
		\underline{\emph{(b)}$\Rightarrow$\emph{(c)}}: We have $F$ is locally maximal hypomonotone at $(\bar x,\bar y)$, it follows that there is $\ell>0$ such that $F+\ell\text{Id}$ is locally maximal monotone at $(\bar x,\bar y+\ell\bar x)$. Then, by taking $\gamma>\ell$ fixed, from Proposition \ref{prop-slmm-lmm} we have that $F+\gamma\text{Id}$ is strongly locally maximal monotone at $(\bar x,\bar y+\gamma\bar x)$ with modulus $\gamma-\ell$. From \cite[Lemma 6.8]{MR4769808} we have $F+\gamma\text{Id}$ is strongly metrically regular and by using Lemma \ref{sum-rule-gd} we have $D(F+\gamma\text{Id})$ is $(\kappa+\gamma)$-lower-definite around $(\bar x,\bar y+\gamma\bar x)$. Then, Lemma \ref{lem-met-reg} says us $\text{reg}(F+\gamma\text{Id};(\bar x,\bar y+\gamma\bar x))\leq \frac{1}{\gamma+\kappa}$. Define $g:= -\gamma\text{Id}$, we have that $g$ is $\gamma$-Lipschitz. Since $\frac{\gamma}{\kappa+\gamma}<1$, from \cite[Lemma 5F.1]{MR3288139} we obtain directly that $F = (F+\gamma\text{Id})+g$ is strongly metrically regular at $(\bar x, \bar y)$ and we have done.
		\\
		\underline{\emph{(c)}$\Rightarrow$\emph{(a)}}: By Lemma \ref{lem-met-reg}, we have $F$ is strongly locally monotone at $(\bar x,\bar y)$ with modulus $\kappa$. Since $F$ is also strongly metrically regular, the maximality follows from \cite[Lemma 6.8]{MR4769808}.\\
		Now, we show the equivalence of \emph{(d)} with the other assertions provided $\H$ is finite-dimensional.\\
    \underline{\emph{(a)}$\Rightarrow$\emph{(d)}}:  From Lemma \ref{slm-low-def}, we have directly that $D_\ast F$ is $\kappa$-lower-definite around $(\bar x,\bar y)$. We use Lemma \ref{local-closed} to get that $\gph F$ is locally closed at $(\bar x,\bar y)$ and finally, the condition $\bar x\in \liminf_{y\to \bar y}F^{-1}(y)$ can be obtained directly since $F$ is strongly metrically regular (see \cite[Lemma 6.8]{MR4769808}), then the assertion \emph{(d)} is verified.\\
    \underline{\emph{(d)}$\Rightarrow$\emph{(c)}}: Since $\gph DF(x,y)\subset \gph D_\ast F(x,y)$ for all $(x,y)\in \gph F$, we have directly that $DF$ is $\kappa$-lower-definite around $(\bar x,\bar y)$. We are going to prove that $F$ is strongly metrically regular at $(\bar x,\bar y)$. For one side, note that if $b\in \mathbb{B}$ and $a\in D_\ast F(\bar x,\bar y)^{-1}(b)$ we have $b\in D_\ast F(\bar x,\bar y)(a)$, then the lower-definiteness implies that $\kappa\|a\|^2\leq \langle a,b\rangle$. The Cauchy-Schwarz inequality yields $\|a\|\leq \frac{1}{\kappa}$, so it proves that $$|D_\ast F(\bar x,\bar y)^{-1}|^+ := \sup_{b\in \mathbb{B}}\sup_{a\in D_\ast F(\bar x,\bar y)^{-1}(b)} \|a\|<\infty,$$ and together with the local closedness of $\gph F$ and $\bar x\in \liminf_{y\to\bar y} F^{-1}(y)$, we have $F$ is strongly metrically regular at $(\bar x,\bar y)$ (see \cite[Theorem 4D.1]{MR3288139} which is especially true in finite-dimensional setting), then it implies that \emph{(c)} holds and we have completed the proof.
	\end{proof}

    \begin{remark}
        Regarding the equivalence between \emph{(a)} and \emph{(b)} in the theorem above, as we previously pointed out, an analogous result based on coderivatives is given in \cite[Theorem 3.4]{MR3485980}. In contrast to our result, note that we require local maximal hypomonotonicity in \emph{(b)}, whereas their result does not impose maximality. In Example \ref{exam1} below, we demonstrate that this assumption is necessary in our case.
    \end{remark}
	
	The following example shows that the statements \emph{(b)}, \emph{(c)} and \emph{(d)} in Theorem \ref{prop-s-m-r} cannot be weakened in order to only require the lower definiteness of $DF$ (or $D_\ast F$) and to obtain the same conclusion.
	\begin{example}\label{exam1}
		Consider $F\colon \R\tto \R$ given by $F(x) = \{x\}$ for $x<0$, $F(x) = \{x+1\}$ for $x>0$ and $F(0) = \{0,1\}$. Note that for all $(x,y)\in \gph F$ and $(u,v)\in \gph D_\ast F(x,y)$, $\langle u,v\rangle\geq \|u\|^2$. However, note that $0\notin \liminf_{y\to 0}F^{-1}(y)$. We can also see that $F$ is not strongly metrically regular at $(0,0)$. Finally, we can see that $F$ is strongly locally monotone at $(0,0)$ but the maximality fails.
	\end{example}
	
	In the above example the property \emph{(c)} fails due to $F$ is not strongly metrically regular, but it is not metrically regular either. Then, it is worth asking if it is possible to change strong metric regularity in \emph{(c)} by a weaker condition. Observe that strong metric subregularity (in finite dimension) is not enough since Lemma \ref{st-met-sub-lemma}. The following example shows that the assumption of metric regularity is not enough either.
	\begin{example}
		Consider $F\colon \R\tto \R$ given by $F(x) = \{x\}$ for $x<0$ and $F(x) = \{x,2x\}$ for $x\geq 0$. At $(0,0)$, $F$ is not locally monotone and we can see that $DF(x,y)$ is $1$-lower-definite for all $(x,y)\in \gph F$. We can also prove that $\forall x,y\in \R$, $d(x;F^{-1}(y))\leq d(y;F(x))$, in particular, $F$ is metrically regular at $(0,0)$, but it is not strongly metrically regular there. 
	\end{example}
	The following simple example shows that some equivalences stated in Theorem \ref{prop-s-m-r} are not satisfied when $\kappa \leq 0$.
	\begin{example}
		Take $F\colon \R\tto \R$ given by $F(x) = \{0\}$ for all $x\in \R$. Note that $F$ is maximal monotone (globally) and $DF$ is  $0$-lower-definite at every point of $\gph F$. However, for all $(x_0,y_0)\in \gph F$, $x_0\notin \liminf_{y\to y_0}F^{-1}(y)$ and $F$ is not strongly metrically regular at any point of $\gph F$. It follows that \emph{(b)}$\Rightarrow$\emph{(c)} and \emph{(b)}$\Rightarrow$\emph{(d)} fail.
	\end{example}
	However, we can establish an analogous equivalence of the assertions \emph{(a)} and \emph{(b)} in Theorem \ref{prop-s-m-r}, where now the constant $\kappa$ is not necessarily positive.
	\begin{corollary}\label{mon-0low-def}
		Let $F\colon \H\tto\H$ be a set-valued mapping, and $(\bar x,\bar y)\in \gph F$. Let $\kappa\in \R$. Then $F-\kappa\text{Id}$ is locally maximal monotone at $(\bar x,\bar y-\kappa\bar x)$ if and only if $DF$ is $\kappa$-lower-definite around $(\bar x,\bar y)$ and $F$ is locally maximal hypomonotone at $(\bar x,\bar y)$.
	\end{corollary}
	\begin{proof}
		The necessity part is straightforward from Lemma \ref{slm-low-def}. So, now we are going to deal with the converse. Observe that the case $\kappa>0$ is done in Theorem \ref{prop-s-m-r}, so we suppose that $\kappa\leq 0$. Since $F$ is locally maximal hypomonotone at $(\bar x,\bar y)$, there is $\ell>0$ such that $F+\ell\text{Id}$ is locally maximal monotone at $(\bar x,\bar y+\ell\bar x)$. From Proposition \ref{prop-slmm-lmm}, we can suppose that $\ell>-\kappa$. From Lemma \ref{sum-rule-gd} we have $F+\ell\text{Id}$ is $(\ell + \kappa)$-lower-definite around $(\bar x,\bar y+\ell\bar x)$, so we can apply Theorem \ref{prop-s-m-r} in order to obtain that $F+\ell \text{Id}$ is strongly locally maximal monotone at $(\bar x,\bar y+\ell \bar x)$ with modulus $\ell+\kappa$. By using Proposition \ref{prop-slmm-lmm}, we have $F-\kappa\text{Id}$ is locally maximal monotone at $(\bar x,\bar y-\kappa\bar x)$, concluding the desired.
	\end{proof}
	
	\section{Characterization of Variational Convexity}
	
	This section is devoted to study how Theorem \ref{prop-s-m-r} acts when the set-valued mapping is the subdifferential of a given function. We are going to prove that it allows to give a second-order characterization of variational convexity, based on first-order analysis of subdifferential mapping. This kind of results has been recently studied in \cite{gfrerer2025} in finite-dimensional setting, by performing different techniques to the ones presented in our work. Furthermore, in \cite{MR4784083} are presented characterizations of variational convexity through first order properties of coderivatives, where subdifferential continuity is assumed  (see Definition \ref{subdif-cont}),  while in \cite[Theorem 5.5]{MR4887480} and \cite[Theorem 4.4]{khanh2025characterizations} are provided similar characterizations based on coderivatives, without subdifferential continuity in finite-dimensional setting. First, we recall some definitions and properties from variational analysis.

    For a given set $S\subset \H$, we define its indicator function as $x\mapsto \delta(x;S)$ where $\delta(x;S) = 0$ if $x\in S$ and $\delta(x;S) = \infty$ if $x\notin S$. 
    
    Consider a function $f\colon \H\to \R\cup\{\infty\}$ and $x\in \dom f$. We say an element $\zeta$ belongs to the \emph{proximal subdifferential} of $f$ at $x$, denoted by $\partial_P f(x)$ if there are $\eta,\sigma>0$ such that $$f(y)\geq f(x) + \langle \zeta,y-x\rangle - \frac{\sigma}{2}\|y-x\|^2, \forall y\in \mathbb{B}_\eta(x).$$
    Moreover, an element $\zeta$ belongs to the \emph{limiting subdifferential} of $f$ at $x$, denoted by $\partial f(x)$, if there are sequences $(x_n)\to x$ and $(\zeta_n)\rightharpoonup \zeta$ such that $f(x_n)\to f(x)$ and $\forall n\in\N:\zeta_n\in \partial_P f(x_n)$.
    
    A given function $f\colon \H\to \Rex$ is said to be $\sigma$-convex for some $\sigma\in\R$ if $f-\frac{\sigma}{2}\|\cdot\|^2$ is convex on $\H$. We also say $f$ is locally $\sigma$-convex at $\bar x\in \dom f$ if there is a convex open neighborhood $U$ of $\bar x$ such that $f+\delta(\cdot;U)$ is $\sigma$-convex. When $\sigma = 0$, we simply omit the dependence of $\sigma$.\\
    The following notions were introduced by Poliquin and Rockafellar in \cite{Poliquin1996}.
	\begin{definition}
		Let $f\colon \H\to\R\cup\{\infty\}$ be a proper and lsc function. Given $\bar x\in \dom \partial f$ and $\bar y\in \partial f(\bar x)$, we say $f$ is prox-regular at $\bar x$ for $\bar y$ if there are some reals $\epsilon,\sigma>0$ such that for all $x,x'\in \mathbb{B}_{\epsilon}(\bar x)$ with $|f(x)-f(\bar x)|<\epsilon$ and all $y\in \mathbb{B}_{\epsilon}(\bar y)$ with $y\in \partial f(x)$ we have
		\begin{equation}\label{pr-def}
			f(x')\geq f(x) + \langle y,x'-x\rangle - \frac{\sigma}{2}\|x'-x\|^2.
		\end{equation}
	\end{definition}
		
		\begin{definition}\label{subdif-cont}
			A proper function $f\colon \H\to\R\cup\{\infty\}$ is said to be subdifferentially continuous at $\bar x$ for $\bar y\in \partial f(\bar x)$ if $(x,y)\mapsto f(x)$ is continuous at $(\bar x,\bar y)$ relative to $\gph \partial f$.
		\end{definition}
	We now define the variational convexity, a concept introduced in \cite{MR3995335} and widely studied lately (see, e.g., \cite{MR4550949, Rockafellar2025-2, Rockafellar2025, MR4784083}).
	\begin{definition}
		An lsc function $f$ will be called \textit{variationally convex} at $\bar{x}$ for $\bar{y} \in \partial f(\bar{x})$ with modulus $\sigma\in\R$ if for some open convex neighborhood $U \times V$ of $(\bar{x}, \bar{y})$ there is a $\sigma$-convex lsc function $\hat{f} \leq f$ on $U$ such that, for some $\epsilon > 0$, 
		\begin{equation*}
			U_\varepsilon \times V \cap \operatorname{gph} \partial f = U \times V \cap \operatorname{gph} \partial \hat{f}
		\end{equation*}
		and
		\begin{equation*}
			\quad f(x) = \hat{f}(x) \text{ for all } (x,y)\in U_\varepsilon \times V \cap \operatorname{gph} \partial f.
		\end{equation*}
		where $U_\epsilon := \{x\in U : f(x)< f(\bar x)+\epsilon\}$. When $\sigma = 0$ we simply say $f$ is variationally convex at $\bar x$ for $\bar y$.
	\end{definition}
	The following lemma establishes the equivalence between variational convexity with a given modulus and the variational convexity of the shifted function (see, e.g., \cite[Theorem 5.3]{MR4784083} for a similar result).
	\begin{lemma}\label{shift-var-convex}
		Consider a proper and lsc function $ f \colon \H \to \R\cup \{\infty\} $. Given $\bar x\in \dom f$, $\bar y\in \partial f(\bar x)$ and $\sigma\in \R$, one has the following assertions are equivalent
		\begin{enumerate}
			\item [(a)] $f$ is variationally convex at $\bar x$ for $\bar y$ with modulus $\sigma$.
			\item [(b)] $f_\sigma := f-\frac{\sigma}{2}\|\cdot-\bar x\|^2$ is variationally convex at $\bar x$ for $\bar y$.
		\end{enumerate}
	\end{lemma}
	\begin{proof}
		Suppose \emph{(a)} holds. When $\sigma = 0$ there is nothing to prove. Suppose then, $\sigma\neq 0$. Then, there is an open convex neighborhood $U\times V$ of $(\bar x,\bar y)$ and a $\sigma$-convex function $\hat f\leq f$ on $U$ such that, for some $\epsilon>0$, 
		\begin{equation}\label{eqn-var-cnv}
			U_\epsilon^f\times V\cap \gph \partial f = U\times V\cap \gph \partial \hat f,	
		\end{equation}
		with $f = \hat f$ on the points where \eqref{eqn-var-cnv} is satisfied. Define $\hat{f}_\sigma := \hat{f}-\frac{\sigma}{2}\|\cdot-\bar x\|^2$, we have $\hat{f}_\sigma$ is convex. The sum rule for limiting subdifferential says us $\partial f_\sigma(x) = \partial f(x) - \sigma(x-\bar x)$ for all $x\in \dom f$. Define $\Pi_\sigma\colon (x,y)\mapsto (x,y+\sigma(x-\bar x))$, then, there is $\delta>0$ such that $W\times Z\subset \Pi_\sigma^{-1}(U\times V)$ where $W = \mathbb{B}_{\delta}(\bar x)$ and $Z = \mathbb{B}_\delta(\bar y)$, then by \eqref{eqn-var-cnv} we have
		\begin{equation*}
			W_\epsilon^f\times Z\cap \gph \partial f_\sigma = W\times Z\cap \gph\partial \hat{f}_\sigma.
		\end{equation*}
		Consider $\mathcal{W} = \mathbb{B}_{r}(\bar x)$ where $r = \min\{\delta,\sqrt{\epsilon/|\sigma|}\}$. The choice of $r$ implies $$\mathcal{W}_{\epsilon/2}^{f_\sigma}=\{x\in \mathcal{W} : f_\sigma(x)<f_\sigma(\bar x) + \epsilon/2\}\subset \{x\in \mathcal{W} : f(x)<f(\bar x) + \epsilon\}.$$ 
		Finally, we can verify that $\mathcal{W}_{\epsilon/2}^{f_\sigma}\times Z\cap \gph f_\sigma = \mathcal{W}\times Z\cap \gph \partial \hat{f}_\sigma$ and it is easy to see that $f_\sigma = \hat{f}_\sigma$ there. Then, \emph{(b)} holds. The converse is quite similar to the argument just given.
	\end{proof}
	
	For $\epsilon>0$, we denote by $\mathscr{F}_\epsilon^f$ to the $f$-attentive localization of $\partial f$ at $(x_0,y_0)\in \gph \partial f$ given by 
	\begin{equation*}
		\mathscr{F}_\epsilon^f \colon x\tto \{y: y\in \partial f(x), \|x-x_0\|<\epsilon, \|y-y_0\|<\epsilon, f(x)<f(x_0)+\epsilon\}.
	\end{equation*}
	The following proposition entails a characterization of prox regular functions in terms of local hypomonotonicity of the $f$-attentive localization of $\partial f$. A similar equivalence is entailed in \cite[Theorem 1]{Rockafellar2025-2}.
	\begin{proposition}\label{prop-prox-reg}
		Consider a proper and lsc function $f\colon \H\to \R\cup\{\infty\}$. Let $x_0\in \dom f$ and $y_0\in \partial f(x_0)$. Then, the following assertions are equivalent
		\begin{enumerate}
			\item [(a)] $f$ is prox regular at $x_0$ for $y_0$.
			\item [(b)] $y_0\in \partial_P f(x_0)$ and there is $\bar\epsilon>0$, such that for all $\epsilon\in ]0,\bar \epsilon[$, $\mathscr{F}_\epsilon^f$ is locally maximal hypomonotone at $(x_0,y_0)$.
			\item [(c)] $y_0\in \partial_P f(x_0)$ and $\mathscr{F}_\epsilon^f$ is locally hypomonotone at $(x_0,y_0)$ for some $\epsilon>0$.
		\end{enumerate}
	\end{proposition}
	\begin{proof}
		First, note that \emph{(b)}$\Rightarrow$\emph{(c)} is straightforward and \emph{(c)}$\Rightarrow$\emph{(a)} follows from 
		\cite[Theorem 11.50]{MR4659163}. It remains to prove \emph{(a)}$\Rightarrow$\emph{(b)}. Indeed, suppose $f$ is prox regular at $x_0$ for $y_0$ with parameters $\hat \epsilon>0$ and $\hat\sigma>0$. From \eqref{pr-def} follows directly that $y_0\in \partial_P f(x_0)$. On the other hand, define 
        \begin{equation*}
            g(x) := f(x+x_0) - f(x_0) -\langle y_0,x\rangle.
        \end{equation*}
        We observe that $g(0) = 0$, $0\in \partial g(0)$ and $g(x)\geq -\frac{\hat\sigma}{2}\|x\|^2$ for all $x\in \mathbb{B}_{\hat\epsilon}(0)$ with $f(x+x_0)<f(x_0)+\hat\epsilon$. If $x\in \mathbb{B}_{\hat\epsilon}(0)$ and $f(x+x_0)\geq f(x_0) + \hat\epsilon$, we have that $g(x) + \langle y_0,x\rangle\geq \hat\epsilon$. Note that $\langle y_0,x\rangle\leq \hat \epsilon + \frac{1+\|y_0\|^2}{4\hat \epsilon}\|x\|^2$ then it follows that $g(x)\geq -(\frac{1+\|y_0\|^2}{4\hat\epsilon})\|x\|^2$, so taking $\alpha := \max\{ \frac{1+\|y_0\|^2}{4\hat\epsilon},\frac{\hat\sigma}{2} \}$ we have 
        \begin{equation*}
            \forall x\in \mathbb{B}_{\hat\epsilon}(0) : g(x) \geq -\alpha\|x\|^2.
        \end{equation*}
        Moreover, note that $\text{Id} + \hat\sigma\mathscr{F}_{\bar \epsilon}^g$ is monotone for some $\bar\epsilon<\hat\epsilon$, and in fact $\text{Id} + \hat\sigma\mathscr{F}_{\epsilon}^g$ is monotone for all $\epsilon\in ]0,\bar\epsilon[$. Take any $\epsilon\in ]0,\bar\epsilon[$, by virtue of \cite[Lemma 11.51]{MR4659163}, for all $\lambda>0$ small enough, $(\text{Id} + \lambda\mathscr{F}_\epsilon^{g})^{-1}$ is single-valued and continuous on a boundary of $0$. We also observe that if $x\in \mathbb{B}_{\epsilon}(x)$ and $g(x)<\epsilon$, then $f(x+x_0)<f(x_0) + \epsilon(1+\|y_0\|)$, hence calculus rules of subdifferential yields $\mathscr{F}_{\epsilon'}^g(x) \subset \mathscr{F}_\epsilon^f(x+x_0)-y_0$, for all $x\in \H$ where $\epsilon' = \frac{\epsilon}{1+\|y_0\|}$. Therefore, 
        \begin{equation*}
            \forall x\in \H :(\text{Id}+\lambda\mathscr{F}_{\epsilon'}^g)^{-1}(x -(x_0+\lambda y_0))+x_0\subset (\text{Id} + \lambda\mathscr{F}_\epsilon^f)^{-1}(x),
        \end{equation*}
 and it follows that for all $\lambda>0$ small enough, $(\text{Id} + \lambda\mathscr{F}_\epsilon^f)^{-1}$ is single-valued and continuous in a boundary of $x_0+\lambda y_0$. From \cite[Theorem 8.5]{MR4769808} it follows that $\mathscr{F}_\epsilon^{f}$ is locally maximal hypomonotone at $(x_0,y_0)$, concluding the desired implication.
	\end{proof}
	
	
	Given $\bar x\in \H$, for a neighborhood $U$ of $\bar x$, $\epsilon >0$ and $f\colon \H\to\R\cup\{\infty\}$, we define $$U_\epsilon^f = \{x\in U : f(x)<f(\bar x) + \epsilon\}.$$
	We also define the $f$-attentive graphical derivative (introduced in \cite{gfrerer2025}) of $f$ at $(x,y)\in \gph \partial f$ in the direction $u$ as
\begin{equation*}
D_f(\partial f)(x,y)(u) =\left\{
  v \;\middle|\;
  \begin{aligned}
  & \exists s_k\searrow 0, \exists (u_k,v_k)\to (u,v) \text{ s.t. }\forall k\in\N,\\
  & y+s_kv_k\in \partial f(x+s_ku_k) \text{ and } f(x+s_k u_k)\to f(x)
  \end{aligned}
\right\}.
\end{equation*}
Observe that when $f$ is subdifferentially continuous, $D_f(\partial f)$ coincides with $D(\partial f)$.\\
We recall that a function $f\colon \H\to \R\cup\{\infty\}$ is said to be sequentially weakly lsc at $x\in \dom f$ if for all $(x_n)\rightharpoonup x$, $\liminf_{n\to \infty} f(x_n)\geq f(x)$.\\
The following theorem can be seen as the infinite-dimensional counterpart of the equivalence \emph{(i)}$\iff$\emph{(iv)} given in \cite[Theorem 4.3]{gfrerer2025}. Our approach consists in applying Theorem \ref{prop-s-m-r} and Corollary \ref{mon-0low-def}, regarding Proposition \ref{prop-prox-reg} since locally maximal hypomonotonicity is required.
	\begin{theorem}\label{char-var-convex}
		Let us consider a proper function $f\colon \H\to \R\cup\{\infty\}$ with $\bar x\in\dom f$, $\bar y\in \partial f(\bar x)$ and $f$ is sequentially weakly lsc around $\bar x$. Let us consider $\kappa\in\R$. Then, the following statements are equivalent 
		\begin{enumerate}
			\item [(a)] $f$ is variationally convex at $\bar x$ for $\bar y$ with modulus $\kappa$.
			\item [(b)] $f$ is prox regular at $\bar x$ for $\bar y$ and there exist $\epsilon>0$ and a neighborhood $U\times V$ of $(\bar x,\bar y)$ such that $D_f(\partial f)$ is $\kappa$-lower-definite on $U_\epsilon^f\times V$.
		\end{enumerate}
	\end{theorem}
	
	\begin{proof}
		\underline{\emph{(a)}$\Rightarrow$\emph{(b)}}: The prox regularity of $f$ is immediate since $f$ is variationally convex. For other side, there exist $\epsilon>0$, a neighborhood $U\times V$ of $(\bar x,\bar y)$ and a $\kappa$-convex and lsc function $\varphi\colon \H\to\R\cup\{\infty\}$ such that $f\geq \varphi$ on $U$, $$U_\epsilon^f\times V\cap \gph \partial f = U\times V\cap \gph \partial \varphi\text{ and } f=\varphi \text{ on } U_\epsilon^f\times V\cap \gph \partial f.$$ 
		Note that if $(x,y)\in U_\epsilon^f\times V\cap \gph \partial f$ and $(u,v)\in \gph D_f(\partial f)(x,y)$, there are $(u_k,v_k)\to (u,v)$ and $s_k\searrow 0$ such that $y+s_k v_k\in \partial f(x+s_ku_k)$ for all $k\in\N$ and $f(x + s_k u_k)\to f(x)$. Then, for all $k\in\N$ big enough, $$(x+s_k u_k,y+s_k v_k)\in U_\epsilon^f\times V\cap\gph \partial f.$$ The $\kappa$-convexity of $\varphi$ implies that for all $k\in\N$ big enough $$\langle x+s_k u_k - x,y+s_k v_k - y\rangle\geq \kappa\|x+s_k u_k - x\|^2,$$
		it follows that $\langle u_k,v_k\rangle\geq \kappa\|u_k\|^2 $, then taking $k\to\infty$ we conclude that $\langle u,v\rangle\geq \kappa\|u\|^2$ and it allows to conclude that $D_f(\partial f)$ is $\kappa$-lower-definite on $U_f^\epsilon\times V$.\\ 
		\underline{\emph{(b)}$\Rightarrow$\emph{(a)}}: Since $f$ is prox regular, we have there is $\bar \epsilon>0$ such that for all $\eta\in ]0,\bar \epsilon[$ the following localization $$\mathscr{F}_\eta^f\colon x\tto \{y : y\in \partial f(x) , \|x-\bar x\|<\eta, \|y-\bar y\|<\eta, f(x)<f(\bar x) + \eta \}$$ is locally maximal hypomonotone at $(\bar x,\bar y)$ (see Proposition \ref{prop-prox-reg}). Moreover, suppose \eqref{pr-def} holds for some $\epsilon'>0$ and $\sigma>0$. From now on, we fix $\eta<\min\{\bar \epsilon, \epsilon, \epsilon'\}$ such that $\mathbb{B}_\eta(\bar x)\times \mathbb{B}_\eta(\bar y)\subset U\times V$. We split this part by proving the following claims:\\
		{\emph{{\bf Claim 1}: For all $(x,y)\in \mathbb{B}_{\eta}(\bar x)\times \mathbb{B}_{\eta}(\bar y)$, $\gph D\mathscr{F}_\eta^f(x,y) \subset \gph D_f(\partial f)(x,y)$. }}\\
		\emph{Proof of Claim 1:} Indeed, let $(u,v)\in \gph D\mathscr{F}_\eta^f(x,y)$, then there are sequences $s_k\searrow 0$ and $(u_k,v_k)\to (u,v)$ such that for all $k\in\N$, $y+s_k v_k\in\mathscr{F}_\eta^f(x+s_k u_k)$. Note that it implies $y+s_k v_k\in \partial f(x+s_k u_k)$, $f(x+s_k u_k)<f(\bar x) + \eta$, $x+s_k u_k\in \mathbb{B}_{\eta}(\bar x)$ and $y+s_k v_k\in \mathbb{B}_{\eta}(\bar y)$ for all $k\in \N$. Then, regarding \eqref{pr-def}, we have for all $k\in\N$ $$f(x)\geq f(x+s_k u_k) + \langle y+s_k v_k, x-(x+s_k u_k)\rangle - \frac{\sigma}{2}\|x-(x+s_k u_k)\|^2.$$
		It follows that,
		$$ f(x)\geq f(x+s_k u_k) - s_k\langle y+s_kv_k,u_k \rangle - \frac{\sigma s_k^2}{2}\|u_k\|^2.$$
		Taking $k\to\infty$ we conclude that $\displaystyle\limsup_{k\to\infty} f(x+s_k u_k) \leq f(x)$ and since $f$ is lsc we conclude $\displaystyle f(x) = \lim_{k\to\infty} f(x+s_k u_k)$, it follows that $(u,v)\in \gph D_f(\partial f)(x,y)$.\\
		\emph{{\bf Claim 2}: $D\mathscr{F}_\eta^f$ is $\kappa$-lower-definite around $(\bar x ,\bar y)$.}\\
		\emph{Proof of Claim 2:} Define $\mathcal{U} := \mathbb{B}_\eta(\bar x)$ and $\mathcal{V} := \mathbb{B}_\eta(\bar y)$. Note that, Claim 1 implies that $D\mathscr{F}_\eta^f$ is $\kappa$-lower-definite on $\mathcal{U}_{\epsilon}^f\times \mathcal{V}$. Since $\eta<\epsilon$, we have $\dom \mathscr{F}_\eta^f\subset \mathcal{U}_{\epsilon}^f$. It follows that $D\mathscr{F}_\eta^f$ is $\kappa$-lower-definite on $\mathcal{U}\times \mathcal{V}$.\\
		\emph{{\bf Claim 3}: For $\kappa\geq 0$, (b)$\Rightarrow$(a) holds.}\\
		\emph{Proof of Claim 3:} Suppose $\kappa>0$. Since $\mathscr{F}_\eta^f$ is locally maximal hypomonotone at $(\bar x,\bar y)$ and $D\mathscr{F}_\eta^f$ is $\kappa$-lower-definite around $(\bar x,\bar y)$, by using Theorem \ref{prop-s-m-r} we have directly that $\mathscr{F}_\eta^f$ is strongly locally maximal monotone with modulus $\kappa$. From \cite[Theorem 6.10]{MR4784083}, since $f$ is also weakly sequentially lsc around $\bar x$, we conclude $f$ is variationally convex at $\bar x$ for $\bar y$ with modulus $\kappa$. When $\kappa = 0$, we proceed similarly, concluding by Corollary \ref{mon-0low-def} that $\mathscr{F}_\eta^f$ is locally maximal monotone at $(\bar x,\bar y)$. Finally, by \cite[Theorem 6.6]{MR4784083}, we conclude $f$ is variationally convex at $\bar x$ for $\bar y$.\\
		\emph{{\bf Claim 4}: (b)$\Rightarrow$(a) holds for $\kappa<0$.}\\
		\emph{Proof of Claim 4:} Consider $g := f-\frac{\kappa}{2}\|\cdot - \bar x\|^2$. First of all, note that since $f$ is weakly sequentially lsc around $\bar x$ and $\frac{-\kappa}{2}\|\cdot-\bar x\|^2$ is convex and continuous, the function $g$ is sequentially weakly lsc around $\bar x$. On the one hand, since $f$ is prox regular at $\bar x$ for $\bar y$, by virtue of Proposition \ref{prop-prox-reg} it is not difficult to see that $g$ is prox regular at $\bar x$ for $\bar y$. On the other hand, by the sum rule for limiting subdifferential, we have $\partial g(x) = \partial f(x) -\kappa(x-\bar x)$ for all $x\in \dom f$. We consider $\Pi_\kappa\colon (x,y)\mapsto (x,y+\kappa(x-\bar x))$. Consider $\delta>0$ such that $\mathbb{B}_\delta(\bar x)\times \mathbb{B}_\delta(\bar y)\subset \Pi_\kappa^{-1}(U\times V)$ and define $\mathscr{U} := \mathbb{B}_r(\bar x)$ and $\mathscr{V} = \mathbb{B}_\delta(\bar y)$. Consider $(x,y)\in \mathscr{U}_{\epsilon}^g\times \mathscr{V}$ and $(u,v)\in \gph D_g(\partial g)(x,y)$, then $$(u,v+\kappa u)\in \gph D_f(\partial f)(x,y+\kappa(x-\bar x))\text{ and }(x,y+\kappa(x-\bar x))\in U\times V.$$ We also have $x\in U_\epsilon^f$ due to $\kappa<0$. Since $D_f(\partial f)$ is $\kappa$-lower-definite on $U_\epsilon^f\times V$, we have $\langle u,v+\kappa u\rangle\geq \kappa\|u\|^2\Rightarrow\langle u,v\rangle\geq 0$. It follows that $D_g(\partial g)$ is $0$-lower-definite on $\mathscr{U}_{\epsilon}^g\times \mathscr{V}$. To finalize, note that $g$ satisfies the assumptions required to apply Claim 3 (with $\kappa=0$), then we have $g$ is variationally convex at $\bar x$ for $\bar y$, thus, applying Lemma \ref{shift-var-convex}, we conclude $f$ is variationally convex at $\bar x$ for $\bar y$ with modulus $\kappa$. 
	\end{proof}
	
	From the last result, we can derive a characterization of local convexity of a $\mathcal{C}^{1,+}$ function, i.e. a continuously differentiable function whose gradient is locally Lipschitz.
	
	\begin{corollary}\label{rmk-c11}
		Let us consider a function $f\colon \H\to \R\cup\{\infty\}$ which is of class $\mathcal{C}^{1,+}$ around a point $\bar x\in \dom f$ and also sequentially weakly lsc there. Let $\kappa\in\R$. Then, the following assertions are equivalent
		\begin{enumerate}
			\item [(a)] $f$ is locally $\kappa$-convex at $\bar x$.
			\item [(b)] $D(\nabla f)$ is $\kappa$-lower-definite around $(\bar x,\nabla f(\bar x))$.
		\end{enumerate}
		
	\end{corollary}
	\begin{proof}
		The function $f$ is prox regular at $\bar x$ for $\nabla f(\bar x)$, as a consequence of combining Proposition \ref{prop-lips-max-hypo} and Proposition \ref{prop-prox-reg}. Note that $f$ is also subdifferentially continuous at $\bar x$ for $\nabla f(\bar x)$. Moreover, if $f$ is variationally convex at $\bar x$ for $\nabla f(\bar x)$ with modulus $\kappa$, then the function $\bar f = f - \frac{\kappa}{2}\|\cdot - \bar x\|^2$ is variationally convex at $\bar x$ for $\nabla f(\bar x)$ (Lemma \ref{shift-var-convex}), and it is still of class $\mathcal{C}^{1,+}$ around $\bar x$, then by \cite[Remark 8.15 (iv)]{MR4769808} we have $\bar f$ is locally convex around $\bar x$ and therefore $f$ is locally $\kappa$-convex around $\bar x$. Then, Theorem \ref{char-var-convex} allows us to conclude the equivalence between \emph{(a)} and \emph{(b)}.
	\end{proof}

    \section{Concluding Remarks}
        In this paper, we have established new characterizations of local maximal monotonicity for set-valued mappings based on the lower-definiteness of the graphical derivative. In particular, Theorem \ref{prop-s-m-r} provides three distinct characterizations of strong maximal monotonicity. Finally, in Theorem \ref{char-var-convex}, we apply our main results to characterize variational convexity, leveraging its connection with the local maximal monotonicity properties of the subdifferential mapping. This can be seen as a second-order characterization of variational convexity, complementing the result in \cite[Theorem 7.3]{MR4784083}, where coderivatives are used and subdifferential continuity is required.
        
        The results obtained strengthen the theory of locally monotone operators and offer potential applications in second-order theory, stability of solution mappings, and related areas. In concrete, some notions of Hölderian and Lipschitzian full stability of parametric variational systems are introduced and characterized in \cite{MR3485980} (see also \cite{MR3881947} for related results), where the connection of these notions with strong local maximal monotonicity is crucial. The solution mapping under consideration corresponds to
        \begin{equation*}
            S(p,v) = \{ x\in \H : v\in f(p,x) + \partial_x g(p,x) \}
        \end{equation*}
        where $f\colon P\times \H\to \R$ and $g\colon P\times \H\to \R\cup\{\infty\}$ are given functions and $P$ represents the space of parameters. In particular, when the parameter $p$ is omitted, the aforementioned concepts of full stability corresponds to strong local maximal monotonicity of $f+\partial g$ (see \cite[Lemma 3.3]{MR3485980}), which can be characterized by using Theorem \ref{prop-s-m-r} in terms of lower definiteness of its graphical derivative. However, the parameterized problem is more challenging and it could be investigated in future works in order to have a graphical derivative counterpart of \cite[Theorem 4.3]{MR3485980}.

\noindent {\bf Acknowledgments} We are grateful to the anonymous referees for their valuable comments and suggestions. Also, we thank Pedro Pérez-Aros and Emilio Vilches for insightful discussions that significantly improved the quality of the manuscript.

\section*{Declarations}
\noindent {\bf Data availability} No data was generated or analyzed.

\noindent {\bf Conflict of interest}  The authors have not disclosed any competing interests.

	\bibliographystyle{plain}
	\bibliography{bib}
	
\end{document}